\documentclass[12pt]{amsart}
\usepackage{epsfig,color}
\usepackage{blindtext}
\usepackage{graphicx}
\usepackage{float}
\usepackage[shortlabels]{enumitem}
\usepackage{url}
\usepackage{amssymb}
\usepackage{graphicx,import}
\usepackage{comment}
\usepackage{esint}
\usepackage{xcolor}
\usepackage{mathtools}
\usepackage{comment}
\usepackage{tikz}
\usepackage{thmtools}
\usepackage{thm-restate}
\usepackage{hyperref}
\usepackage{cleveref}

\declaretheorem[name=Lemma,numberwithin=section]{lemm}

\usepackage[margin = 1in] {geometry}

\usepackage{hyperref}
\usepackage{dsfont}

\newtheorem{theorem}{Theorem}[section]

\newtheorem{proposition}[theorem]{Proposition}
\newtheorem{lemma}[theorem]{Lemma}

\theoremstyle{definition}

\newtheorem{definition}[theorem]{Definition}
\newtheorem*{remark*}{Remark}

\newtheorem*{theorem*}{Theorem}
\newtheorem{remark}[theorem]{Remark}

\numberwithin{equation}{section}

\newcommand{\R}{\mathbb{R}}

\newcommand{\N}{\mathbb{N}}
\newcommand{\Z}{\mathbb{Z}}

\newcommand{\eps}{\varepsilon}

\newcommand{\length}{\mathrm{length}}

\DeclareMathOperator{\dmn}{domain}
\DeclareMathOperator{\sing}{sing}
\newcommand{\red}[1]{\textcolor{red}{#1}}

\newcommand{\purple}[1]{\textcolor{purple}{#1}}
\newcommand{\nl}{\newline}

\newcommand{\p}{\partial}

\newcommand{\RR}{\mathbb{R}}
\renewcommand{\P}{\mathbb{P}}
\newcommand{\met}{\text{Met}}
\newcommand{\cS}{\mathcal{S}}
\newcommand{\cI}{\mathcal{I}}

\newcommand{\cV}{\mathcal{V}}
\newcommand{\stplus}{(S^2)^+}
\newcommand{\embeds}{\hookrightarrow}
\newcommand{\cF}{\mathcal{F}}
\newcommand{\cZ}{\mathcal{Z}}
\newcommand{\bF}{\mathbf{F}}

\title{The p-widths of the Hemisphere}
\author{Jared Marx-Kuo}
\begin{document}

\begin{abstract}
We compute the p-widths, $\{\omega_p\}$, for the hemisphere with the standard round metric. This provides the first example of a manifold with boundary for which the $p$-widths are known for all $p$.
\end{abstract}
\maketitle
\tableofcontents
\section{Introduction}
%
In \cite{gromov2006dimension}, Gromov introduced the \textit{$p$-widths}, $\{\omega_p\}$, of a Riemannian manifold, $(M^{n+1}, g)$, as an analogue of the spectrum of the Laplacian, $\{\lambda_k\}$, the discrete collection of eigenvalues. Intuitively, one replaces the Rayleigh quotient in the definition of eigenvalues with the area functional in the definition of $\omega_p$ - see \cite{gromov2010singularities} \cite{gromov2002isoperimetry} \cite{guth2009minimax} \cite{fraser2020applications} for more background. The $p$-widths have proven to be extremely useful due to the Almgren--Pitts/Marques--Neves Morse theory program for the area functional (see \cite{almgren1965theory}\cite{pitts2014existence}\cite{marques2014min} \cite{marques2023morse} \cite{marques2015morse} \cite{marques2021morse} \cite{almgren1962homotopy} \cite{zhou2020multiplicity}). In \cite{liokumovich2018weyl}, Liokumovich--Marques--Neves demonstrated that $\{\omega_p\}$ satisfy a Weyl law, which later led to a resolution of Yau's conjecture \cite{yau1982problem}: Any closed three-dimensional manifold must contain an infinite number of immersed minimal surfaces. In some sense, the resolution was stronger than expected:
\begin{theorem}[\cite{chodosh2020minimal} \cite{song2018existence} \cite{marques2019equidistribution} \cite{marques2017existence} \cite{irie2018density} \cite{zhou2020multiplicity}] \label{yauThm}
On any closed manifold $(M^{n+1}, g)$ with $3 \leq n + 1 \leq 7$, there exist infinitely many embedded minimal hypersurfaces.
\end{theorem}
\noindent We refer the reader to the introduction of \cite{marx2024isospectral} for more information. In theorem \ref{yauThm}, the dimension upper bound is due to the existence of singular minimal hypersurfaces when $n+1 \geq 8$ (however, see an analogous theorem by Li \cite{li2023existence}), while the lower bound is due to the fact that min-max on surfaces may only produce stationary geodesic networks in general (see \cite{pitts1973net}, \cite[Remark 1.1]{marques2015morse}). \newline 
\indent In a breakthrough \cite{chodosh2023p}, Chodosh--Mantoulidis used the Allen--Cahn equation to prove that the $p$-widths can always be realized as the sum of lengths of a union of closed geodesics. Building on this, Chodosh--Cholsaipant \cite[Thm 1.3]{chodosh2025p} proved an analogous theorem for polygonal domains in $\R^2$, as well as surfaces with geodesic boundary.
\begin{theorem}[\cite{chodosh2025p}] \label{ChoChoRegularity}
For $(M^2, g)$ a riemannian surface with geodesic boundary and $p = 1,2, \dots$, there is a finite set of billiard trajectories, $\{\gamma_{p,i}\}_{i = 1}^{N_p}$ (possibly with repetition) such that 
\begin{equation} \label{billiardeqn}
\omega_p(M^2, g) = \sum_{i = 1}^{N_p} \ell(\gamma_{p,i})
\end{equation}
\end{theorem}
While the p-widths have been useful in their ability to show the existence of many minimal surfaces, they are notoriously difficult to compute for all $p$. Chodosh--Mantoulidis provided the first example of knowing all $p$-widths for a fixed manifold:
\begin{theorem}[\cite{chodosh2023p}, Thm 1.4] \label{SphereWidths}
For $g_0$ the unit round metric on $S^2$
\[
\forall p \in \N, \qquad \omega_p(S^2, g_0) = 2 \pi \lfloor \sqrt{p} \rfloor
\]
\end{theorem}
In \cite{marx2024isospectral}, the author showed that a connected component of Zoll metrics on $S^2$ have the same values of $\{\omega_p\}$:
\begin{theorem} \label{IsoSpectralThm}
Let $g$ be any Zoll metric on $S^2$ which lies in the connected component of $g_{0}$ in the space of Zoll metrics on $S^2$. Then $\omega_p(g) = \omega_p(g_{0}) = 2 \pi \lfloor \sqrt{p} \rfloor$ for all $p$.
\end{theorem}
The author also extended the computation of theorem \ref{SphereWidths} to the real projective plane in \cite{marx2025p}.
\begin{theorem} \label{rptwowidths}
For $\overline{g_0}$ the quotient of the unit round metric on $\R \P^2$
\[
\forall p \in \N, \qquad \omega_p(\R \P^2, \overline{g_0}) = 2 \pi \Big\lfloor \frac{1}{4} (1 + \sqrt{1 + 8 p} )\Big\rfloor
\]   
\end{theorem}
\indent In related work, Ambrozio--Marques--Neves \cite{ambrozio2024rigidity}, showed that $(\R \P^2, \overline{g_0})$ is p-width spectrally \textit{rigid}, i.e.
\begin{theorem}[Thm A, \cite{ambrozio2024rigidity}]
Let $(M^{n+1}, g)$ a compact riemannian manifold. If 
\[
\omega_p(M, g) = \omega_p(\R \P^2, \overline{g_0})
\]
for all $p \in \mathbb{N}^+$, then $(M, g)$ is isometric to $(\R \P^2, \overline{g_0})$.
\end{theorem}
We also mention the work of Donato \cite{donato2022first} who computes the first few widths of the unit disk, as well as Aiex \cite{aiex2016width}, who computed the first few widths of ellipses, laying some of the foundation for theorem \ref{SphereWidths}. \nl 
\indent While computing p-widths has been more accessible on surfaces, there are several questions about the parameter $p$ and how it corresponds to geometric features of the underlying curves/billiard trajectories, as in equation \eqref{billiardeqn}. We mention related work by the author, Sarnataro, and Stryker, connecting the number of vertices and geometric index of the geodesics in \eqref{billiardeqn} to the number $p$ \cite{marx2024index}.
\subsection{Main Result + Ideas}
In this paper, we compute the full p-width spectrum for $\stplus$ with the standard metric, $g_{std}$, coming from restricting the unit round metric on $S^2$.
\begin{theorem} \label{mainTheorem}
For $g_{std}$ the standard metric on the hemisphere, $\stplus$, we have 
\[
\omega_p(\stplus, g_{std}) = \pi \cdot \Big\lfloor \frac{1}{2} \left(-1 + \sqrt{1 + 8p} \right) \Big\rfloor \quad \forall p \in \mathbb{N}^+
\]
Moreover, $\omega_p$ is achieved by sweepouts coming from polynomials in $x$ and $y$.
\end{theorem}
\noindent The proof is an adapation of the work done by Chodosh--Mantoulidis \cite{chodosh2023p} for the p-widths of $S^2$, with the following modifications:
\begin{enumerate}
    \item Applying the regularity theory of Chodosh--Cholsaipant \cite{chodosh2025p} to show that 
    \[
    \omega_p = \sum_{i = 1}^{N_p} \ell(\Gamma_{i,p})
    \]
    where $\Gamma_{i,p}$ is a closed billiard trajectory (see \S \ref{ACSection} for definition).
    
    \item Considering sweepouts by level sets of polynomials of the form $f(x,y)$ of degree at most $d$. These can be thought of as the $z$-independent subset of polynomials of the form $f(x,y) + zg(x,y) \in \R^3[x,y,z]/(x^2 + y^2 + z^2 = 1)$, which were the original set of polynomial sweepouts on $S^2$.
    
    \item Perturbing the standard metric on the hemisphere to a metric $g_{\mu}$ so that $\{z = 0\}$ corresponds to a closed geodesic of length $2 \pi + \mu$ and $\{x = 0\}, \{y = 0\}$ correspond to free boundary geodesics of lengths $\pi, \pi + \mu$. The other natural choices (informed by Chodosh--Mantoulidis' \cite{chodosh2023p} original construction) of $\{z = 0\}$ having length $2 \pi$ or $2 \pi + 2 \mu$ are not compatible with the counting argument of section \S \ref{ProofMain}.

    \item Proving that for the perturbed metrics, $\omega_p(g_{\mu}) < \omega_{p+1}(g_{\mu})$ for $p$ fixed and $\mu$ sufficiently small. This requires a slightly technical modification of the analogous proposition 6.11 of \cite{chodosh2023p}, accounting for the fact that we have a surface with boundary.
\end{enumerate}
%
%
In \S \ref{mmBackground}, we define the basic notions of min-max theory as relevant to this paper. In \S \ref{lemmasSection}, we recall some essential propositions and lemmas from \cite{chodosh2023p} and \cite{chodosh2025p}. We also construct the modified sweepouts to attain an upper bound for $\omega_p(\stplus)$. In \S \ref{perturbSection}, we prove the perturbative inequality $\omega_p(g_{\mu}) < \omega_{p+1}(g_{\mu})$, and in \S \ref{ProofMain}, we prove the main theorem using the lemmas and a modified counting argument.

\subsection{Acknowledgements}
The author would like to thank Christos Mantoulidis for helpful discussions. The author was supported by NSF grant $23-603$.
\section{Min--Max Background and Notation} \label{mmBackground}
We recall the essential notation from the min-max theory of Almgren--Pitts \cite{almgren1962homotopy} and Marques--Neves \cite{MN2020CETSurvey}. Let $M$ be an $(n+1)$-dimensional manifold with boundary, $\p M$. Let $X \subseteq [0,1]^k$ denote a cubical subcomplex of the $k$-dimensional unit cube (cf. \cite[\S 2.2]{MN2020CETSurvey}), and let $\mathcal{Z}_n(M, \partial M; \Z_2)$ denote the set of relative flat $n$-chains mod $\Z_2$ - informally, such a chain is the boundary of a nice (Caccioppoli) open set where we ignore any component of $\partial \Omega$ contained inside of $\partial M$. In \cite{almgren1962homotopy}, Almgren showed that $\mathcal{Z}_n(M, \p M; \Z_2)$ is weakly homotopic to $\mathbb{R} \mathbb{P}^{\infty}$. We let $\overline{\lambda} \in H^1(\mathcal{Z}_n(M, \p M; \Z_2), \Z_2) \cong \Z_2$ denote the generator of this group, so that $\overline{\lambda}^p$ generates $H^p(\mathcal{Z}_n(M, \p M; \Z_2), \Z_2) \cong \Z_2$. For $T \in \mathcal{Z}_n(M, \p M; \Z_2)$, there is a natural associated varifold to $T$ and we let $||T||$ denote the corresponding Radon measure, with $||T||(U)$ denoting the integral of the measure over $U$. We further let $\mathbf{M}(T) = ||T||(M)$. 
\begin{definition}[\cite{marques2017existence}, Defn 4.1] \label{defi:sweepout}
	A map $\Phi : X\to \mathcal{Z}_{n}(M, \p M;\Z_{2})$ is a \textbf{$p$-sweepout} if it is continuous (with the standard flat norm topology on $\mathcal{Z}_n(M, \p M; \Z_2)$) and $\Phi^{*}(\overline\lambda^{p}) \neq 0 \in H^p(X)$. 
\end{definition}

\begin{definition}[\cite{marques2017existence} \S 3.3] \label{defi:no.concentration.of.mass}
	A map $\Phi : X \to \mathcal{Z}_n(M, \p M; \Z_2)$ is said to have \textbf{no concentration of mass} if
	\[ \lim_{r \to 0} \sup \{ \Vert \Phi(x) \Vert(B_r(p)) : x \in X, p \in M \} = 0. \]
\end{definition}

\begin{definition}[\cite{gromov2002isoperimetry}, \cite{guth2009minimax}] \label{defi:p-width}
	We define $\mathcal{P}_{p} = \mathcal{P}_{p}(M)$ to be the set of all $p$-sweepouts, out of any cubical subcomplex $X$, with no concentration of mass. The \textit{$p$-width} of $(M,g)$ is
	\[
	\omega_{p}(M,g) := \inf_{\Phi \in \mathcal{P}_{p}} \sup\{\mathbf{M}(\Phi(x)) : x \in \mathrm{domain}(\Phi)\}.
	\]
\end{definition}
\noindent We note that one can work with refined sweepouts to compute $\omega_p$, i.e. let 
\[ 
\mathcal{P}^{\mathbf{F}}_{p,m} : = \{\Phi \in \mathcal{P}_{p} : \dmn(\Phi) \subset I^m \text{ and } \Phi \text{ is } \mathbf{F}\text{-continuous} \} 
\]
then:
\begin{lemm}[Lemma 2.6, \cite{chodosh2023p}] \label{lemm:P.p.m}
If $m = 2p+1$, then
\[
\omega_{p}(M,g) = \inf_{\Phi \in \mathcal{P}^{\mathbf{F}}_{p,m}} \sup\{\mathbf{M}(\Phi(x)) : x \in \dmn(\Phi)\}.
\]
\end{lemm}
\subsection{Allen--Cahn theory} \label{ACSection}
We remark that there is another way to compute the p-widths using the Allen--Cahn theory. In the below, we consider $(M^2, g)$ a smooth surface with boundary. The Allen--Cahn energy (with the sine gordon potential) is given by 
\begin{align*}
u & \in H^1(M)  \\
E_{\eps}(u) &= \int_M \frac{\eps}{2} |\nabla^g u|^2 + \frac{W(u)}{\eps} \\
W(u) &:= \frac{1 + \cos(\pi u)}{\pi^2}
\end{align*}
and critical points satisfy the equation 
\begin{align} \label{ACeqn}
\eps^2 \Delta_g u &= W'(u) \\ \label{neumann}
\p_{\nu} u \Big|_{\p M} &= 0
\end{align}
%
%
\indent Let $\tilde{P}_p$ denote the space of $\Z / 2 \Z$-equivariant maps $h: \tilde{X} \to H^1(M) \backslash \{0\}$, where $\tilde{X}$ is the double cover of some $p$-dimensional subcomplex of $I^m$. We define
\[
\omega_{p,\eps} = \inf_{h \in \tilde{P}_p} \max_{x \in \text{Dom}(h)} E_{\eps}(h(x))
\]
We then have the following of equality of widths in the limit as $\eps \to 0$
\begin{proposition}[\cite{dey2022comparison}, \cite{gaspar2018allen}] \label{widthequality}
Let $h_0$ be the energy of the one dimensional heteroclinic solution to \eqref{ACeqn} on $\R$. Then 
\[
\lim_{\eps \to 0} h_0^{-1} \omega_{p,\eps}(M^2, g) = \omega_p(M^2, g)
\]
\end{proposition}
In fact, proposition \ref{widthequality} holds for any homotopy class $\Pi$ - we refer the reader to \cite{chodosh2025p} for more details. Chodosh--Cholsaipant also prove the following regularity theory for surfaces with totally geodesic boundary. 
\begin{theorem}[Thm 1.3 \cite{chodosh2025p}] \label{ChoChoRegularity}
For $(M^2, g)$ a riemannian surface with geodesic boundary and $p = 1,2, \dots,$, there's a finite set of billiard trajectories, $\{\gamma_{p,i}\}_{i = 1}^{N_p}$ (possibly with repetition) such that 
\[
\omega_p(M^2, g) = \sum_{i = 1}^{N_p} \ell(\gamma_{p,i})
\]
\end{theorem}
\noindent We note that in addition to the homotopy class corresponding to the $p$-widths, theorem \ref{ChoChoRegularity} holds for any homotopy class, $\Pi$. \nl 
\indent We recall that on a smooth surface with geodesic boundary, a \textbf{billiard trajectory, $\gamma$}, is a either a closed geodesic, a free boundary geodesic, or a finite union of geodesic segments, $\gamma_1, \dots, \gamma_k$, such that the endpoint of $\gamma_i$, $p$, agrees with the starting point of $\gamma_{i+1}$ for $1\leq i \leq k-1$, and $\gamma_i'(p)$ is the reflection of $\gamma_{i+1}'(p)$ across $T \partial M$.
\begin{remark} \label{doublingRemark}
Let $(M^2, g)$ be a manifold with geodesic boundary, so that the double 
\[
\tilde{M} = M \sqcup_{\partial M} M = \{(x, \eta) \; | \; x \in M, \; \eta \in \{-1,1\}\} / ((x, -1) \sim (x, + 1)) \; \forall x \in \partial M)
\]
is a smooth topological manifold. Let $\Pi(x,\eta) = x$, and suppose that the pull back metric, $\tilde{g}\Big|(x, \eta) = \Pi^*(g) \Big|_{x}$ is smooth. Then for $\gamma$, a closed billiard trajectory on $M$, $\Pi^{-1}(\gamma)$ is a $\Z/2\Z$-invariant (with respect to the map $R: \tilde{M} \to \tilde{M}$ action given by $R(x, \pm 1) = (x \mp 1)$) union of closed immersed geodesics on $\tilde{M}$.
\end{remark}
\section{Fundamental Lemmas} \label{lemmasSection}

\subsection{A good metric on $(S^2)^+$}
We recall some relevant theorems from \cite[\S 6]{chodosh2023p}: Consider the ellipsoids
\begin{equation} \label{ellipsoidEq}
E(a_1, a_2, a_3) := \{ (x_1, x_2, x_3) \in \RR^3 : a_1 x_1^2 + a_2 x_2^2 + a_3 x_3^2 = 1 \} \subset \RR^3 
\end{equation}
and the three geodesics
\[ \gamma_i(a_1, a_2, a_3) := E(a_1, a_2, a_3) \cap \{ x_i = 0 \}, \; i = 1, 2, 3 \]
on them. It is shown in \cite[Theorems IX 3.3, 4.1]{Morse:calculus.variations} that for every $\Lambda > 2\pi$, if $a_1 < a_2 < a_3$ are sufficiently close to $1$ (depending on $\Lambda$), then every closed connected immersed geodesic $\gamma \subset E(a_1, a_2, a_3)$ (coverings allowed) satisfies:
\begin{equation} \label{eq:geodesic.network.ellipsoid.geodesic.genericity}
	\gamma \text{ has no nontrivial normal Jacobi fields if } \length(\gamma) < 2\Lambda,
\end{equation} 
and
\begin{equation} \label{eq:geodesic.network.ellipsoid.geodesic.uniqueness}
	\gamma \text{ is an iterate of one of } \gamma_i(a_1, a_2, a_3), \; i = 1, 2, 3, \text{ if } \length(\gamma) < 2\Lambda.
\end{equation}
We recall the following theorem:
\begin{theorem}[Thm 6.1, \cite{chodosh2023p}]\label{theo:good.metric.final}
	Let $\Lambda > 0$ and $U$ be any neighborhood, in the $C^\infty$ topology, of the standard metric $g_{std} \in \met(S^2)$. There is a $\mu_0 = \mu_0(\Lambda, U) > 0$, so that for all $\mu \in (0, \mu_0)$, there exists $g_\mu \in U$ with all these properties:
	\begin{enumerate}
		\item There are three simple closed geodesics $\gamma_1, \gamma_2, \gamma_3 \subset (S^2, g_\mu)$ so that $\length_{g_\mu}(\gamma_i) = \pi + (i-1) \mu$.
		\item If a closed connected geodesic in $(S^2, g_\mu)$ has $\length_{g_\mu} < \Lambda$, then it is an iterate of $\gamma_i$, $i = 1$, $2$, $3$.
		\item There are no not-everywhere-tangential stationary varifold Jacobi fields along any $S \in \cS^{\Lambda}(g_\mu)$.
	\end{enumerate}
	Moreover, $g_\mu \to g_{std}$ as $\mu \to 0$ in the $C^\infty$ topology.
\end{theorem}
Note in particular, that $g_{\mu}$ is the induced metric from $E(a_1(\mu), a_2(\mu), a_3(\mu))$ for smooth functions $a_i(\mu)$, and hence these metrics have reflectional symmetry about the three standard axes. We now consider the hemi-ellipsoids
\begin{equation} \label{hemiellipsoidEq}
E^+(a_1(\mu), a_2(\mu), a_3(\mu)) := \{ (x_1, x_2, x_3) \in \RR^3 : a_1 x_1^2 + a_2 x_2^2 + a_3 x_3^2 = 1, x_3 \geq 0\} \subset \RR^3 
\end{equation}
and the analogous curves
\[ \gamma_i^+(a_1, a_2, a_3) := E^+(a_1, a_2, a_3) \cap \{ x_i = 0 \}, \; i = 1, 2, 3 \]
where by convention
\[
\ell(\gamma_i^+(a_1(\mu), a_2(\mu), a_3(\mu)) = \begin{cases} 
\pi + (i-1) \mu & i \in \{1,2\} \\
2 \pi + \mu & i = 3
\end{cases}
\]
Note that $\gamma_1, \gamma_2$ are now free boundary geodesics while $\gamma_3$ is a closed geodesic. For $E^+$ as above, we label the induced metric on the hemisphere as $(\stplus, g_{\mu})$. By remark \ref{doublingRemark} and theorem \ref{theo:good.metric.final}, we have the following proposition:
\begin{proposition} \label{propbilliardrigid}
For the same parameters, $\Lambda, \mu_0$ as in theorem \ref{theo:good.metric.final}, any closed billiard trajectory of length $< \Lambda/2$ on $(\stplus, g_{\mu})$ is one of $\{\gamma_i\}$.
\end{proposition}
\subsection{Upper Bound}
We recreate the analogous upper bound on the p-widths, by using a restricted class of sweepouts on $\stplus$. First, let $D(d) = \begin{pmatrix}
    (d+2) \\ 2
\end{pmatrix}$, and let $f: \Z^+ \to \Z^+$ be such that 
\[
f(p) = d \iff p \in \{D(d-1), \dots, D(d)-1\}
\]
\begin{lemma} \label{lem:perturbedMetrics}
For $p \in \N^*$, there exists $\mu_1 > 0$ and an open neighborhood $U$ of the unit round metric on $\stplus$ depending on $p$ such that for all $\mu \in (0, \mu_1)$, $g_{\mu} \in U$ and 
\begin{enumerate}
    \item \label{enum:upperBound} $\omega_p(\stplus, g_{\mu}) \leq \pi f(p) + 1$
    \item \label{enum:containment} For any $\mathcal{F}$-homotopy class $\Pi \subseteq \mathcal{P}^{\mathbf{F}}_{p,m}$ (i.e. $p$-sweepouts coming from domains of dimension at most $m$)
    \begin{align*}
    \mathbf{L}_{AP}(\Pi, g_{\mu}) &\in  \Big( \{ n_1 \cdot \pi + n_2 \cdot (\pi + \mu) + n_3 \cdot (2\pi + \mu) \; : \; (n_1,n_2,n_3) \in \N^3 \} \backslash \{0\}\Big) \\
    & \qquad \cup [\pi f(p) + 2, \infty)
    \end{align*}
    %
\end{enumerate}
\end{lemma}
\begin{proof}
To verify \ref{enum:upperBound}, let 
\[
A_{d} = \{f(x,y) \; | \; \deg(f) \leq d\}
\]
Note that 
\begin{align*}
\dim(A_{d}) &= \begin{pmatrix}
    d+2 \\ 2
\end{pmatrix} = \frac{(d+2)(d+1)}{2}
\end{align*}
Let $\{p_i^d\}$ be an enumeration of polynomials of degree $\leq d$. For $d \geq 1$, we define a $D(d)-1$-sweepout of $(\stplus, g_{std})$ as follows
\begin{align*}
F_d&: \R \P^{D(d) - 1} \to \mathcal{Z}_1( \stplus, \p \stplus; \Z_2) \\
F_d([a_1 : \cdots : a_{D(d)}]) &= \Big\{ \sum_{i = 1}^{(d+1)(d+2)/2} a_{i} p^{d}_{i}(x,y) = 0\Big\} \cap \stplus
\end{align*}
%
%
%
We see that $F_d$ is $D(d)-1$-sweepout by checking that it is a $1$-sweepout when restricting to the copy of $\R \P^1$ realized by $0 = a_3 = \dots a_{D(d)}$ as then
\[
F_d([a_1: a_2: 0 : \cdots : 0]) = \{a_1 + a_2 x= 0\} \cap \stplus
\]
where by convention, we assume that $p_1^{d} = 1$ and $p_2^{d} = x$. \newline 
\indent We now compute an upper bound on $\mathbf{M}(F)$ using the crofton formula, i.e.
\begin{align*}
\mathbf{M}(F_d([\vec{a}])) &= \frac{1}{2} \mathbf{M}(\tilde{F}_d([\vec{a}])) = \frac{1}{2} \frac{1}{4} \int_{\xi \in S^2} \sharp (\tilde{F}([\vec{a}]) \cap \xi^{\perp})
\end{align*}
where $\tilde{F}_d([\vec{a}]) = \Pi^*(F_d([\vec{a}])) \in \mathcal{Z}_1( S^2, \Z_2)$ is the reflection of the current across the x-y plane (recall we are working with the unit round metric still), and $\xi^{\perp}$ denotes the great circle which lies in the plane in $\R^3$ normal to $\xi$. Note that every such plane is given by $\{\alpha x + \beta y + \gamma z = 0\}$ for some $(\alpha, \beta, \gamma) \in \mathbb{R}^3$. Then 
\begin{align*}
\tilde{F}_d([\vec{a}]) \cap \xi^{\perp} &= \Big\{ \sum_{i = 1}^{(d+1)(d+2)/2} a_{i} p^{d}_{i}(x,y) = 0\Big\} \cap \{x^2 + y^2 + z^2 = 1\} \\
& \qquad \quad \cap \{\alpha x + \beta y + \gamma z = 0\}
\end{align*}
By Bezout's inequality, we have that 
\[
\sharp \left(\tilde{F}_d([\vec{a}]) \cap \xi^{\perp}\right) \leq d \cdot 2 \cdot 1 = 2d
\]
%
%
Thus, we conclude 
\begin{align*}
\mathbf{M}(F_d([\vec{a}])) &\leq \frac{1}{2} \frac{1}{4} \int_{\xi \in S^2} 2d \\
& \leq \frac{1}{8} \cdot 4 \pi \cdot 2d \\
&= \pi \cdot d
\end{align*}
%
%
Now notice that if $f(p) = d$ then $F_{d}$ is a p-sweepout, and so we conclude 
\[
\omega_p(\stplus, g_{std}) \leq \pi f(p)
\]
and by continuity of the p-widths (see \cite[Lemma 2.1]{irie2018density}), we have that if $U$ is a sufficiently small neighborhood of $(\stplus, g_{std})$, then for any $g \in U$,
\[
\omega_p(\stplus, g) \leq \pi f(p) + 1
\]
Similarly, given such a neighborhood $U$, we can apply theorem \ref{theo:good.metric.final} to find $g_{\mu} \in U$ for all $\mu  < \mu_0(\pi f(p) + 2, U)$ and the above bound will also hold for $g = g_{\mu}$. \nl 
\indent To verify part \ref{enum:containment}, fix the metric $g_{\mu}$ as above and choose an $\mathcal{F}$ homotopy class with value $L_{AP}(\Pi) < \pi f(p) + 2$. By theorem \ref{ChoChoRegularity} applied to the homotopy class $\Pi$, we have that 
\[
L_{AP}(\Pi) = \sum_{j = 1}^N \ell(\Gamma_{j})
\]
for $\Gamma_j$ a closed billiard trajectory. By proposition \ref{propbilliardrigid}, each $\Gamma_j$ is one of $\{\gamma_i\}_{i=1}^3$, so we conclude that
\[
L_{AP}(\Pi) = m_1 \cdot \pi + m_2 \cdot (\pi + \mu) + m_3 \cdot (2 \pi + \mu) = (n_1 + n_2 + 2 n_3) \pi + \mu (n_2 + n_3), \; n_i \in \Z^{>0}
\]
\end{proof}
\section{Perturbing the Widths} \label{perturbSection}
\noindent The goal of this section is to prove the following adaptation of Prop 6.11, \cite{chodosh2023p}. 
\begin{proposition} \label{widthInequality}
Fix $p \in \mathbb{N}$ and $\mu_1(p) > 0$ as in lemma \ref{lem:perturbedMetrics}. Then 
\[
\omega_p(\stplus, g_{\mu}) < \omega_{p+1}(\stplus, g_{\mu})
\]
for every $\mu \in (0, \mu_1)$.
\end{proposition}
This section may be skipped on a first pass of the paper. The proof of proposition \ref{widthInequality} is a modification of the analogous proposition in \cite{chodosh2023p} (see also \cite[\S 6]{marques2017existence} or \cite[Appendix A]{aiex2016width}), however there are a few complications due to the fact that $\stplus$ is a manifold with geodesic boundary, and in the course of the proof, we pass between working with $\mathcal{Z}_1(\stplus, \partial \stplus; \Z_2)$ and $\mathcal{Z}_1(S^2; \Z_2)$ to apply the prior work of Chodosh--Mantoulidis \cite{chodosh2023p}. \nl 
\indent We recall a few concepts from \cite{chodosh2023p, li2021min} (see also \cite{donato2022first}) adapted to our setting. Let $M$ be a smooth surface with (potentially empty) geodesic boundary, $\partial M$, so that the topological double with the pullback metric is a smooth, closed Riemannian manifold. Let $\mathcal{X}(M)$ denote the set of all smooth vector fields on $M$ and $\mathcal{X}_{Tan}(M)$ the subset of vector fields so that $X \Big|_{\partial M} \in T \partial M$. We say that $V \in \cI \cV_1(M)$ is $g$-stationary with free boundary if 
\[
\delta V(X) = 0 \qquad \forall X \in \mathcal{X}_{Tan}(M)
\]
\begin{definition} \label{defi:geodesic.network.stationary.lambda}
For $g \in \met^k(M)$, $\Lambda > 0$, we define:
\[ \cS^\Lambda(g) := \{ S \in \cI\cV_1(M) \text{ is } g\text{-stationary with free boundary and } \# \sing S + \Vert S \Vert(M, g) < \Lambda \}. \]
\[ \overline{\cS}^\Lambda(g) := \{ S \in \cI\cV_1(M) \text{ is } g\text{-stationary with free boundary  and } \# \sing S + \Vert S \Vert(M, g) \leq \Lambda \}. \]
\end{definition}
\noindent where $\cI \cV_1(M)$ denotes the set of integral $1$-varifolds. When $M$ is closed, the phrase ``with free boundary" can be omitted as in \cite[Defn 5.1]{chodosh2023p}. For $M$ with a smooth double, $\tilde{M}$, any element $V \in \cS^{\Lambda}(g)$ lifts to $\Pi^*(V) \in \cS^{\Lambda}(\tilde{g})$ (for $\Pi$ as in remark \ref{doublingRemark}), a $\tilde{g}$-stationary integral on $\tilde{M}$. Regular elements of $S^{\Lambda}(g)$ are known as \textit{free boundary geodesic networks} (cf. \cite[Defn 3.1]{donato2022first}), and in our setting, these are exactly the integral varifolds which pull back to $\Pi$-invariant geodesic networks on $\tilde{M}$. \nl 
%
\indent We prove a technical min-max lemma which we will use in the proof of proposition \ref{widthInequality} (compare Prop 2.9, \cite{chodosh2023p})
\begin{proposition}\label{prop:am.vary.dom}
Fix $k \in \N^*$ and assume that 
\[
\Phi_i : Y_i\to\cZ_1(M, \p M ;\bF;\Z_2)
\]
is a sequence of continuous maps for $Y_i$ cubical subcomplexes of $I^k$ so that every $V\in\mathbf{C}(\{\Phi_i\})$ is stationary with free boundary in $M$. Then, at least one of the following holds:
\begin{enumerate}
	\item $\mathbf{C}(\{\Phi_i\})$ contains an integral varifold that is stationary with free boundary and has $\leq 5^k$ singular points
	\item There exists a sequence of continuous $\Psi_i^* : Y_i \to \cZ_1(M, \p M; \bF;\Z_2)$, with each $\Psi_i^*$ homotopic to $\Psi_i$ in the $\cF$-topology, so that
\[
\mathbf{L}_\textrm{AP}(\{\Psi_i^*\}) <  \mathbf{L}_{\textrm{AP}}(\{\Psi_i\}).
\]
\end{enumerate}
\end{proposition}
\begin{proof}
The proof follows that of Prop 2.9, \cite{chodosh2023p}, though we include more details to describe the modifications needed to work in $\cZ_1(M, \p M; \bF; \Z_2)$, most of which are accounted for by work of Li--Zhou \cite{li2021min}. Our goal is to show that at least one of the following hold:
\begin{enumerate}
\item[(1')] There is some $V \in \mathbf{C}(\{\Phi_i\})$ with the property that for any $5^k$ distinct points $\{p_j\}_{j=1}^{5^k}$ with minimal pairwise distance $d$, it holds that $V$ is $\Z_2$-almost minimizing with free boundary (see \cite[Defn 2.10]{donato2022first} or \cite[Defn 3.18]{li2021min}) in at least one of $\{B_{d/16}(p_j)\}_{j=1}^{5^k}$
\item[(2')] Conclusion (2) above holds.
\end{enumerate}
\indent By \cite[Lemma 2.5]{chodosh2023p}, the maps $\{\Phi_i\}$ are continuous in the $\mathcal{F}$ norm and have no concentration of mass. We can apply the discretization and interpolation processes of \cite[Thm 4.12, 4.14]{li2021min} (as opposed to those of Marques--Neves \cite[Prop 3.1]{marques2021morse} in the closed setting) to find homotopic maps which are continuous in the mass topology, which we also label as $\{\Phi_i\}$. \nl 
\indent Suppose that (1') does not hold. We proceed with the strategy of Li \cite[Lemma 3.2]{li2023existence}: via our assumption, for each $V \in \mathbf{C}(S)$, there exist $5^k$ concentric annuli $A(p_j^V, r_j^V - s_j^V, r_j^V + s_j^V)$ such that $V$ is not almost minimizing with free boundary in any one of the annuli. Applying \cite[Lemma 3.1]{li2023existence}, we create discrete sweepouts $S_{\sharp} = \{\psi_i\}$ with $\mathbf{C}(S_{\sharp}) = \mathbf{C}(S)$ but $L(S_{\sharp}) < L(S)$, after using the deformation procedure and corresponding boundary-relevant lemmas and theorems from \cite[Thm 4.21]{li2021min}. \nl 
\indent Now the interpolation theorem of Li--Zhou (\cite[Thm 4.14]{li2021min}) produces the corresponding Almgren extensions, $\{\Psi_i^*\}$, which are now continuous and homotopic to $\{\Psi_i^*\}$, with $L_{AP}(\{\Psi_i^*\}) < L_{AP}(\{\Psi_i\})$, proving (2). \nl 
\indent Thus, WLOG, we can assume that (1') holds. It follows from \cite{pitts1981existence, allard1976structure} that $\Pi^*(V)$ is a stationary geodesic net, and hence $V$ itself is a free boundary geodesic network with $\sharp \; \text{sing} \; V \leq 5^k$.
\end{proof}
\begin{proof}[Proof of proposition \ref{widthInequality}]
We sketch the details in comparison to \cite[Prop 6.11]{chodosh2023p}: \nl 
\indent Fix $p \in \mathbb{N}$ and suppose for contradiction that 
\[
\omega_p(\stplus, g_{\mu}) = \omega_{p+1}(\stplus, g_{\mu})
\]
for all $\mu$ sufficiently small. By \cite[Cor 6.3]{chodosh2023p} (the proof adapts to the boundary case in a straightforward manner), there exists $X \subseteq I^{2p+3}$ with a corresponding homotopy class $\Pi$ so that 
\[
L_{AP}(\Pi, g_{\mu}) = \omega_{p+1}(\stplus, g_{\mu})
\]
For $\Lambda' = L_{AP}(\Pi) + 5^{2p +3}$, $\Lambda = \Lambda' + 1$ we define 
\[
\mathcal{S}^* = \{V \in \overline{\mathcal{S}}^{2 \Lambda'}(g_{\mu}) \; : \; ||V|| \leq L_{AP}(\Pi, g_{\mu})\} \subseteq \cI \cV_1(S^2)
\]
In comparison to the original proof, we've doubled the upper bound, used $\leq$ instead of $=$ in our definition of $\cS$, and are considering integral varifolds \textit{on $S^2$} as opposed to \textit{on $\stplus$}. Note that $\cS^*$ is still compact in the $\mathbf{F}$ topology. Since $\Lambda$ is bounded (by e.g. the Weyl law), by point (3) of \ref{theo:good.metric.final}. we have (by (3) of \cite[Thm 5.17]{chodosh2023p}) that there is $N = N(p)$ such that 
\[
Lim^{(N)} \mathcal{S}^* =\emptyset
\]
i.e. there are no $N$th iterated limit points of $\cS$. Now we consider 
\[
\mathcal{S} = \{V \in \overline{S}^{\Lambda'}(\stplus, g_{\mu}) \; | \; ||V|| = L_{AP}(\Pi, g_{\mu})\}
\]
and note that $\mathcal{S} \embeds \mathcal{S}^*$ via the map $V \to \Pi^*(V)$. This tells us that $Lim^{(N)} \mathcal{S} =\emptyset$ as well and also that $\mathcal{S}$ is compact in the $\mathbf{F}$ topology, since it embeds as a closed subset of $\cS^*$. \nl
\indent Fix $\eps > 0$ so that every continuous map $\Phi: S^1 \to \cZ_1(M, \partial M; \Z_2)$ with $\Phi(S^1) \subseteq B_{\eps}^{\cF}(T)$ for some $T \in \Z_1(M, \p M; \Z_2)$ is homotopically trivial. This is a direct adaptation of \cite[Prop 3.3]{marques2017existence} to the boundary setting, noting that Marques--Neves adapt the original work of Almgren, whose construction applies to relative flat chains (see \cite[Remark 8.3]{almgren1962homotopy}). \nl 
\indent We similarly note that the entirety of \cite[\S 6.2]{chodosh2023p} applies to manifolds with boundary, simply replacing $\cZ_1(M; \Z_2)$ with $\cZ_1(M, \p M, \Z_2)$ and lifting integral varifolds on $M$ to integral varifolds on $\tilde{M}$ to apply the regularity theory of Allard--Almgren \cite{allard1976structure} where needed. This allows us to apply the analogues of \cite[Lemma 6.6]{chodosh2023p}, \cite[Cor 6.8]{chodosh2023p} to conclude that there exist $\{V_1, \dots, V_k\} \subset \cS$, $\eta_1, \dots, \eta_k, \eps_1, \dots, \eps_k > 0$ such that 
\begin{equation} \label{eq:LS.arg.cS.cover}
\cS \subset \cup_{j = 1}^k B_{\eta_j}^{\mathbf{F}}(V_j)
\end{equation}
and for any continuous $\Phi: S^1 \to \cZ_1(M, \p M; \Z_2)$ satisfying
\begin{align*}
|\Phi(t)| \subset \cup_{j = 1}^k B_{2 \eta_j}^{\mathbf{F}}(V_j) & \quad  \forall t \in S^1 \\
\implies \Phi(S^1) \subseteq B_{\eps}^{\mathcal{F}}(T) & \quad \text{for some } T \in \mathcal{T}(\{V_1, \dots, V_k\})
\end{align*}
We now obtain a contradiction as follows: choose a sequence $\{\Phi_i\}_{i=1}^{\infty} \subseteq \Pi$ which has been pulled tight so that every element of $\mathbf{C}(\{\Phi_i\})$ is stationary with free boundary (see \cite[Thm 2.2]{donato2022first} or \cite[4.17]{li2021min}), which by an analogous interpolation theorem (see \cite[Thm 4.12, 4.14]{li2021min}), can be assumed to be continuous w.r.t. the mass topology. \nl 
\indent For $\{\ell_i\}_{i=1}^\infty\subset \N$ to be chosen, define $Y_i \subset X(\ell_i)$ to be the subcomplex consisting of cells $\alpha \in X(\ell_i)$ so that 
\[
|\Phi_i(x)| \not \in \bigcup_{j=1}^k B_{\eta_j}^\bF(V_j)
\]
for every vertex $x\in\alpha_0$. We fix $\ell_i$ sufficiently large so that (i) if $x \in \overline{X\setminus Y_i}$, then
\[
|\Phi_i(x)| \in \bigcup_{j=1}^k B_{2\eta_j}^\bF(V_j)
\]
and (ii) the fineness of $\Phi_i$ restricted to $X(\ell_i)_0$ is less than $1/i$.  \nl 
\indent By the same argument as in Chodosh--Mantoulidis: $\Psi_i: = (\Phi_i)|_{Y_i}$ are $p$-sweepouts for $i$ sufficiently large, Thus,
\begin{align*}
\omega_p(\stplus,g_\mu) \leq \limsup_{i\to\infty}\sup_{x\in Y_i} \mathbf{M}(\Psi_i(x)) &\leq \omega_{p+1}(\stplus,g_\mu) = \omega_p(\stplus,g_\mu) \\
\implies \limsup_{i \to \infty} \sup_{x \in Y_i} \mathbf{M}(\Psi_i(x)) &= \omega_p(\stplus,g_\mu).
\end{align*}
Because $\Psi_i$ are $p$-sweepouts, it follows that alternative (2) of Proposition \ref{prop:am.vary.dom} cannot occur. Thus, by alternative (1), there must exist $V \in \mathbf{C}(\{\Psi_i\})\cap \cS^{\Lambda}(g_\mu) \subset \cS$. By (ii) in the choice of $\ell_i$ and \cite[p. 66]{pitts1981existence}, we can pass to a subsequence and find $x_i \in (Y_i)_0 \subset X(\ell_i)_0$ with 
\[
\lim_{i\to\infty}\bF(|\Phi_i(x_i)|,V) = 0.
\]
On the other hand, since $x_i \in (Y_i)_0$, we have $|\Phi_i(x_i)| \not \in \bigcup_{j=1}^k B_{\eta_j}^\bF(V_j)$ for all $i$, so
\[
V \not \in \bigcup_{j=1}^k B_{\eta_j}^\bF(V_j).
\]
This contradicts \eqref{eq:LS.arg.cS.cover}, completing the proof. 

\end{proof}
\section{Proof of Main Theorem \ref{mainTheorem}} \label{ProofMain}
\noindent We now compute all of the widths of $(\stplus, g_{std})$.
\begin{theorem*}
For $(\stplus, g_{std})$, we have for any $p \in \mathbb{N}^+$,
\[
\omega_p = 2 \pi f(p) = \Big\lfloor \frac{-1 + \sqrt{1 + 8 p}}{2} \Big\rfloor
\]
\end{theorem*}
\begin{proof}
As in the proof of theorem \ref{SphereWidths} from \cite{chodosh2023p}, we will show that 
\[
\omega_{D(d)} = \omega_{D(d) + 1} = \cdots = \omega_{D(d+1) - 1} = \pi (d+1)
\]
for all $d$. Let $\mu_1 = \mu_1(p = D(d+1) - 1) > 0$ as in lemma \ref{lem:perturbedMetrics} and suppose that $0 < \mu < \min(\mu_1, 1/2(d+1))$. \nl 
\indent As in \cite[Claim 7.1]{chodosh2023p}, we want to prove:
\begin{align*}
\{\omega_p(\stplus, g_{\mu}) \; : \; p = 1, \dots, D(d+1) - 1\} & = \Big( \{ n_1 \cdot \pi + n_2 \cdot (\pi + \mu) + n_3 \cdot (2\pi + \mu) \; \\
    &  \quad : \; (n_1,n_2,n_3) \in \N^3 \} \backslash \{0\},  \Big) \cap (0, \pi (d+1) + 1]
\end{align*}
Denote the set on the left as $L$ and on the right as $R$. $R \subseteq L$ follows directly from lemma \ref{lem:perturbedMetrics}. For the other containment, we note that by proposition \ref{widthInequality}:
\begin{align*}
\sharp \{\omega_p(\stplus, g_{\mu}) \; : \; p = 1, \dots, D(d+1) - 1\} & = D(d+1) - 1 = \frac{(d+1)(d+4)}{2}
\end{align*}
To compute $|R|$, we argue as in \cite[Claim 7.1]{chodosh2023p} and write
\begin{align*}
n_1 \cdot \pi + n_2 \cdot (\pi + \mu) + n_3 \cdot (2\pi + \mu) &= \pi \cdot (n_1 + n_2 + 2n_3) + \mu (n_2 + n_3)
\end{align*}
And count the number of values of $\{\mu (n_2 +n_3) \; | \; n_1 + n_2 + 2n_3 = j\}$ where $0 \leq j \leq d+1$. However
\[
\{\mu (n_2 + n_3) \; | \; n_1 + n_2 + 2n_3 = j\} = j+1
\]
and so 
\begin{align*}
|R| &= \sum_{j = 1}^{d+1} \sharp \{ \mu (n_2 + n_3) \; | \; n_1 + n_2 + 2n_3 = j\} \\
&= \sum_{j = 1}^{d+1} j+1 \\
&= \frac{(d+1)(d+2)}{2} + d+1 \\
&= \frac{(d+1)(d+4)}{2}
\end{align*}
And so necessarily we have $L = R$. \newline 
\indent Now by induction and the fact that the widths are increasing (proposition \ref{widthInequality}), we see that 
\begin{align*}
\forall p \in \{D(d), D(d) + 1, \dots, D(d+1) - 1\} \\
2\pi \cdot (d+1) \leq \omega_p(\stplus, g_{\mu}) \leq (2\pi + \mu ) \cdot (d+1)
\end{align*}
using continuity of the p-widths in the metrics and noting that $g_{\mu} \xrightarrow{\mu \to 0} g_{std}$, we conclude that 
\begin{equation} \label{widthEquation}
p \in \{D(d), D(d) + 1, \dots, D(d+1) - 1\} \implies \omega_p(\stplus, g_{std}) = 2 \pi \cdot (d+1) = 2\pi \cdot f(p)
\end{equation}
These values are achieved precisely by the sweepouts defined in \ref{lem:perturbedMetrics} by the upper bound computed there and equation \eqref{widthEquation}. The fact that $f(p) = \lfloor \frac{1}{2} \left(-1 + \sqrt{1 + 8p} \right)\rfloor$ follows from algebraic manipulation.
\end{proof}

\bibliography{main}{}
\bibliographystyle{amsalpha}
\end{document}